\documentclass[a4paper,oneside,11pt]{amsart}

\reversemarginpar

\usepackage[colorlinks,hyperindex,linkcolor=blue,urlcolor=black,pdftitle={Invertibility of functions of operators and existence of hyperinvariant subspaces},pdfauthor={Maria F. Gamal'}]
{hyperref}

\usepackage{amsmath}
\usepackage{amsfonts}
\usepackage{amssymb}
\usepackage{amsthm}

\usepackage[english]{babel}
\usepackage{enumerate}

\usepackage{graphicx}
\usepackage{color}

\usepackage[abbrev]{amsrefs}

\numberwithin{equation}{section}

\theoremstyle{plain}
\newtheorem{theorem}{Theorem}[section]
\newtheorem{lemma}[theorem]{Lemma}

\theoremstyle{definition}
\newtheorem{remark}[theorem]{Remark}
\newtheorem{example}[theorem]{Example}

\begin{document}

\title[Invertibility of functions of operators]{Invertibility of functions of operators and existence of hyperinvariant subspaces}

\author{Maria F. Gamal'}
\address{
 St. Petersburg Branch\\ V. A. Steklov Institute 
of Mathematics\\
 Russian Academy of Sciences\\ Fontanka 27, St. Petersburg\\ 
191023, Russia  
}
\email{gamal@pdmi.ras.ru}



\begin{abstract} Let $T$ be an absolutely continuous   polynomially bounded operator, and let $\theta$ be a singular inner function. It is shown that if $\theta(T)$ is invertible and some additional conditions are fulfilled, then $T$ has nontrivial 
hyperinvariant subspaces. 
\end{abstract}

\maketitle

\section{Introduction}

Let $\mathcal H$ be a (complex, separable, infinite dimensional) Hilbert space, and let  $\mathcal L(\mathcal H)$ be the algebra of all (linear, bounded) operators acting on  $\mathcal H$. A (closed) subspace  $\mathcal M$ of  $\mathcal H$ is called \emph{invariant} 
for an operator  $T\in\mathcal L(\mathcal H)$, if $T\mathcal M\subset\mathcal M$, and $\mathcal M$ is called \emph{hyperinvariant} 
for $T$ if $A\mathcal M\subset\mathcal M$ for all $A\in\mathcal L(\mathcal H)$ such that $AT=TA$.  The complete lattice of all invariant (resp.,  hyperinvariant) subspaces of $T$ is denoted by  $\operatorname{Lat}T$ (resp., by 
$\operatorname{Hlat}T$). 
 The hyperinvariant subspace problem is the question whether $\operatorname{Hlat}T$ is nontrivial 
(i.e., $\operatorname{Hlat}T\neq\{\{0\},\mathcal H\}$) whenever $T$ is not a scalar multiple of the identity operator $I$.

 The operator $T\in\mathcal L(\mathcal H)$ is called  \emph{polynomially bounded}, if there exists a constant $M$ such that 
$$\|p(T)\|\leq M \sup\{|p(z)|:  |z|\leq 1\} $$ for every (analytic) polynomial $p$. For every polynomially bounded operator $T\in\mathcal L(\mathcal H)$ 
there exist $\mathcal H_a$, $\mathcal H_s\in\operatorname{Hlat}T$ such that $ \mathcal H=\mathcal H_a\dotplus\mathcal H_s$, $T|_{\mathcal H_s}$ is similar to a 
singular unitary operator, and $T|_{\mathcal H_a}$ has an $H^\infty$-functional calculus, that is, $T|_{\mathcal H_a}$ is an 
 \emph{absolutely continuous (a.c.)} polynomially bounded operator  \cite{mlak}, \cite{ker16}. 
Thus, if $\mathcal H_s\neq\{0\}$ and $T\neq cI$,  then $T$ has nontrivial  hyperinvariant subspaces. For the existence of invariant subspaces of polynomially 
bounded operators see \cite{rej}.

The operator $T\in\mathcal L(\mathcal H)$  is called 
a  \emph{contraction}, if $\|T\|\leq 1$. 
Every contraction is polynomially bounded with  constant 
$1$ by the von Neumann inequality (see, for example, {\cite[Proposition I.8.3]{sznagy}}). 
(It is well known that the converse is not true,  
see \cite{pis} for  the first example of a polynomially bounded operator which is not similar to a contraction.)

It is well known that if the spectrum $\sigma(T)$  of an operator $T$ is not connected, 
then nontrivial  hyperinvariant subspaces of  $T$ can be found by using 
the Riesz--Dunford functional calculus, see, for example {\cite[Theorem 2.10]{rara}}. 
A similar method can be applied, if an operator $T$ has sufficiently rich spectrum and 
appropriate estimate for the norm of the resolvent \cite{apostol}, {\cite[Sec. 4.1]{chalpart}}. 
Using such a method, we show that an a.c. polynomially bounded operator $T$  has 
 nontrivial  hyperinvariant subspaces, if $\theta(T)$ is invertible for a singular inner function $\theta$ and  some
 additional conditions are fulfilled (Sec. 2). Moreover, in this case $T$ cannot be  quasianalytic. 
Examples of quasianalytic operators $T$  for which $\theta(T)$ is invertible for some inner functions $\theta$ are given 
in Sec. 3. (See the beginning of Sec. 3 for references about quasianalyticity.)
In Sec. 4 it  is shown that  if $\varphi(T)$ is invertible for some weighted  shift $T$ and function $\varphi$, 
then $T$ is similar to the simple bilateral shift.

Symbols $\mathbb D$, $\operatorname{clos}\mathbb D$, and $\mathbb T$ denote the open unit disc, the closed unit disc, 
and the unit circle, respectively. The normalized Lebesgue measure on $\mathbb T$ is denoted by $m$.

For any finite, positive, singular (with respect to $m$) Borel measure $\mu$ on $\mathbb T$ define a function  $\theta_\mu$ by the formula 
\begin{equation}\label{thetamu}\theta_\mu(z)=\exp\int_{\mathbb T}\frac{z+\zeta}{z-\zeta}\text{d}\mu(\zeta)
 \ \ \ \ \   \ (z\in\mathbb D).\end{equation}
 The function $\theta_\mu$ is a singular inner function. Recall that $\theta_\mu$ has nontangential boundary values 
equal to zero a.e. with respect to $\mu$ (see, for example, {\cite[Theorem II.6.2]{garn}}).
 As usual, the Dirac measure at a point $\zeta$ is denoted by $\delta_\zeta$.

For $\lambda\in\mathbb D$  set 
\begin{equation}\label{drzeta} D(\lambda)=\{z\in\mathbb C\ :\ |z-\lambda|<1-|\lambda|\}.\end{equation}

The following simple lemmas are given for convenience of references. 

\begin{lemma}\label{adelta} Let $0<r<1$, and let $\zeta\in\mathbb T$. Then 
 \begin{align*} \frac{1-|z|^2}{|\zeta-z|^2}&\leq\frac{r}{1-r}, \ \ \  \text{ if }  \ z\in\mathbb D\setminus D(r\zeta), \\
\text{ and } \   \frac{1-|z|^2}{|\zeta-z|^2}&>\frac{r}{1-r}, \ \ \  \text{ if }  \ z\in D(r\zeta). \end{align*}\end{lemma}

\begin{lemma}\label{lemmaee} Suppose that $E\subset\mathbb T$ is a non-empty closed set, $m(E)=0$,
$\mathbb T\setminus E=\cup_{n\in\frak N}\mathcal J_{\zeta_{1n},\zeta_{2n}}$, where $\frak N\subset \mathbb N$, 
$\mathcal J_{\zeta_{1n},\zeta_{2n}}$ is the open subarc of $\mathbb T$ with endpoints $\zeta_{1n}$,$\zeta_{2n}$,  
$\mathcal J_{\zeta_{1n},\zeta_{2n}}\cap\mathcal  J_{\zeta_{1l},\zeta_{2l}}=\emptyset$ for $n\neq l$. 
Furthermore, let $0<r<1$, and let $\mu$ be a finite positive Borel measure on $E$. 
Then $$|\theta_\mu(z)|\geq\exp\frac{r\mu(E)}{r-1}\ \ \text{ for every } 
z\in\mathbb D\setminus \bigcup_{n\in\frak N,\  k=1,2}D(r\zeta_{kn}).$$
\end{lemma}

\begin{lemma}\label{lemmabb}
Let $\lambda\in\mathbb D\setminus\{0\}$. Set $b_\lambda(z)=\frac{|\lambda|}{\lambda}\frac{\lambda-z}{1-\overline\lambda z}\ \  (z\in\mathbb D).$ Then 
\begin{equation*}|b_\lambda(z)|\geq\frac{1}{3}\ \ \text{ for every } z\in\mathbb D\setminus D(\lambda).
\end{equation*}\end{lemma}

\section{Estimates of the resolvent}

The main result of this section (Theorem~\ref{T:4.8}) is the following. Let $T$ be an a.c. polynomially bounded operator, and  let $\theta$ be a singular inner function. If $\theta(T)$ is invertible, then, under some additional conditions,  $T$ has nontrivial hyperinvariant subspaces.

 Suppose that $\Gamma$ is a simple closed curve, 
 \begin{equation}\label{E:4.1}\begin{gathered}\Gamma\cap\mathbb T=\{\zeta_{\Gamma 1},\zeta_{\Gamma 2}\}, \ \ \zeta_{\Gamma 1}\neq\zeta_{\Gamma 2}, \ \text{ and there exists } \varepsilon>0\ \text{ such that  } \\ 
\Gamma\cap\{z\in\mathbb C: |z-\zeta_{\Gamma l}|\leq \varepsilon, \ |z|\leq 1\} \  \text{ and } \ 
\Gamma\cap\{z\in\mathbb C: |z-\zeta_{\Gamma l}|\leq \varepsilon, \ |z|\geq 1\} \\ 
\text{ are segments nontangential to } \mathbb T\ \text{ at } \zeta_{\Gamma l} \ \ (l=1,2). \end{gathered} \end{equation}
Let $\Omega_\Gamma$ be the  bounded component of $\mathbb C\setminus\Gamma$. 

The following lemmas are well known, see, for example, \cite{apostol}, {\cite[Lemma 3.1]{berc}}, {\cite[Lemma 6]{tak}},
 {\cite[Sec. 4.1]{chalpart}}. We give proofs to emphasize some details.

\begin{lemma}\label{L:4.1} Suppose that $\Gamma$ and $\Gamma '$ are two simple closed rectifiable curves, 
$\Gamma$ and $\Gamma '$ satisfy \eqref{E:4.1},  and 
 \begin{align}\label{E:4.2}\Gamma\setminus\{\zeta_{\Gamma 1},\zeta_{\Gamma 2}\}\subset\mathbb C\setminus(\Gamma '\cup \Omega_{\Gamma '}), \ \  \Gamma '\setminus\{\zeta_{\Gamma '1},\zeta_{\Gamma '2}\}\subset\mathbb C\setminus(\Gamma\cup \Omega_{\Gamma}). \end{align} 
Suppose that $T\in\mathcal L(\mathcal H)$ has the following properties:
\begin{enumerate}[\upshape (i)]
\item there exist $C>0$ and $k\in\mathbb N$ such that $\|(T-\lambda I)^{-1}\|\leq C/\bigl|1-|\lambda|\bigr|^k$ for 
all $\lambda\in\Gamma\setminus\{\zeta_{\Gamma 1},\zeta_{\Gamma 2}\}$ and all 
$\lambda\in\Gamma '\setminus\{\zeta_{\Gamma '1},\zeta_{\Gamma '2}\}$;
\item $\sigma(T)\cap\Omega_\Gamma\neq\emptyset$ and $\sigma(T)\cap\Omega_{\Gamma '}\neq\emptyset$.
\end{enumerate}
Then there exist $\mathcal M$, $\mathcal M '\in\operatorname{Hlat}T$ such that $\mathcal M\neq\{0\}$, $\mathcal M '\neq\{0\}$, 
$\sigma(T|_{\mathcal M})\subset\mathbb C\setminus\Omega_\Gamma$,  and 
$\sigma(T|_{\mathcal M '})\subset\mathbb C\setminus\Omega_{\Gamma '}$.\end{lemma}
\begin{proof} Set $p(\lambda)=(\lambda-\zeta_{\Gamma 1})^k(\lambda-\zeta_{\Gamma 2})^k$, 
$q(\lambda)=(\lambda-\zeta_{\Gamma ' 1})^k(\lambda-\zeta_{\Gamma ' 2})^k$, $\lambda\in\mathbb C$. 
Since the segments defined in \eqref{E:4.1}  are  nontangential to $ \mathbb T$ at $ \zeta_{\Gamma l}$, 
$\zeta_{\Gamma' l}$ $\  (l=1,2)$, we have that 
  \begin{align}\label{E:4.3}\sup_{\lambda\in\Gamma\setminus\{\zeta_{\Gamma 1},\zeta_{\Gamma 2}\}}\!|p(\lambda)|\|(T-\lambda I)^{-1}\|<\infty,   
\sup_{\lambda\in\Gamma ' \setminus\{\zeta_{\Gamma ' 1},\zeta_{\Gamma ' 2}\}}\!|q(\lambda)|\|(T-\lambda I)^{-1}\|<\infty. 
\end{align}
Put
$$ A=\frac{1}{2\pi i}\int_{\Gamma}p(\lambda)(T-\lambda I)^{-1}\text{\rm d}\lambda, \ \ \ 
A'=\frac{1}{2\pi i}\int_{\Gamma '}q(\lambda)(T-\lambda I)^{-1}\text{\rm d}\lambda.$$
  Let $\{0\}\neq\mathcal E\in\operatorname{Hlat}T$.
Let ${\mathcal A}_{\mathcal E}\subset\mathcal L(\mathcal E)$ be a maximal commutative Banach algebra  such that $T|_{\mathcal E}$, 
$(T|_{\mathcal E}-\lambda I_{\mathcal E})^{-1}\in{\mathcal A}_{\mathcal E}$ for all $\lambda\in\mathbb C\setminus\sigma(T|_{\mathcal E})$.
We infer from \eqref{E:4.3} that $A|_{\mathcal E}$, $A'|_{\mathcal E}\in{\mathcal A}_{\mathcal E}$.
Let $\xi\in\Omega_\Gamma$. If $\xi\in\sigma(T|_{\mathcal E})$, then there exists a continuous algebra homomorphism 
$\phi\colon{\mathcal A}_{\mathcal E}\to\mathbb C$ such that $\phi(T|_{\mathcal E})=\xi$. Then 
 \begin{align}\label{E:4.4}\phi(A|_{\mathcal E})=\frac{1}{2\pi i}\int_{\Gamma}p(\lambda)(\xi-\lambda)^{-1}\text{\rm d}\lambda\neq 0. \end{align}
Applying \eqref{E:4.4} with $\mathcal E=\mathcal H$, we conclude that $A\neq\mathbb O$. Similarly, $A'\neq\mathbb O$. 

Put $\mathcal M =\ker A$ and  $\mathcal M '=\ker A'$. Since $A$ and $A'$ commute with all $B\in\mathcal L(\mathcal H)$ such that $BT=TB$, we have $\mathcal M$, $\mathcal M '\in\operatorname{Hlat}T$. 
We prove that $AA'=A'A=\mathbb O$ exactly as in {\cite[Lemma 3.1]{berc}} or {\cite[Sec. 4.1]{chalpart}}, and we obtain from the latter  equality that 
$\mathcal M\neq\{0\}$ and $\mathcal M '\neq\{0\}$. 

Let $\xi\in\Omega_\Gamma$. If $\xi\in\sigma(T|_{\mathcal M})$, then, by \eqref{E:4.4} applied with $\mathcal E=\mathcal M$, 
there exists an algebra homomorphism 
$\phi\colon{\mathcal A}_{\mathcal M}\to\mathbb C$ such that $\phi(A|_{\mathcal M})\neq 0$,
a contradiction with the definition of $\mathcal M$. Therefore, $\sigma(T|_{\mathcal M})\cap\Omega_\Gamma=\emptyset.$
Similarly, $\sigma(T|_{\mathcal M '})\cap\Omega_{\Gamma '}=\emptyset.$ \end{proof}

\begin{lemma} \label{L:4.2} Suppose that  $C>0$, $k\in\mathbb N$,  and an operator $T$ 
is given such that $\sigma(T)\subset\operatorname{clos}\mathbb D$ and $\|(T-\lambda I)^{-1}\|\leq C/(|\lambda|-1)^k$ for all 
$\lambda\notin\operatorname{clos}\mathbb D$. 
Put $$\Lambda =\sigma(T)\cup\bigl\{\lambda\in\mathbb D: \ \ \|(T-\lambda I)^{-1}\|\geq C/(1-|\lambda|)^k\bigr\}$$
and $$\mathcal Z=\bigl\{\zeta\in\sigma(T)\cap\mathbb T: \ \ \sup\{r\in [0,1): r\zeta\in\Lambda\}<1\bigr\}.$$ 
If $\mathcal Z$ is uncountable, then there exist $\mathcal M_1$, $\mathcal M_2\in\operatorname{Hlat}T$ such that $\mathcal M_1\neq\{0\}$, 
$\mathcal M_2\neq\{0\}$, and $\sigma(T|_{\mathcal M_1})\cap\sigma(T|_{\mathcal M_2})=\emptyset$.\end{lemma}
\begin{proof} For $\zeta\in\mathbb T$ and $a$, $b\in (0,1)$ put $$F(\zeta,a,b)= \{r\zeta: a\leq r<1\}\cup\{z\in\mathbb D : \ |z-a\zeta|<b\}.$$
Since $\mathcal Z$ is uncountable, there exist $a$, $b\in (0,1)$ and $\xi_j\in\mathcal Z$, $j=1,\ldots,10$, 
such that $\xi_j$ are arranged on $\mathbb T$  in counter clockwise order, 
$$\cup_{j=1}^{10} F(\xi_j,a,b)\subset\mathbb D\setminus\Lambda \ \ \text{ and } \cap_{j=1}^{10}F(\xi_j,a,b)\neq\emptyset.$$
We construct four simple closed rectifiable curves $\Gamma_1$, $\Gamma_1 '$, $\Gamma_2$, $\Gamma_2 '$ 
such that they satisfy \eqref{E:4.1} (they cross $\mathbb T$ along radial segments),   
$$(\Gamma_1\cup\Gamma_1 '\cup\Gamma_2\cup\Gamma_2 ')\cap\mathbb D\subset\cup_{j=1}^{10} F(\xi_j,a,b),$$
\begin{equation*}\begin{aligned}&\zeta_{\Gamma_1 1}=\xi_4, \ \ \ \zeta_{\Gamma_1 2}=\xi_{10}, \ \ \ \ \ \ & \zeta_{\Gamma_1 '1}=\xi_1, \ \ \ \zeta_{\Gamma_1 '2}=\xi_3, \\ &
\zeta_{\Gamma_2 1}=\xi_5,  \ \ \ \zeta_{\Gamma_2 2 }=\xi_9, \ \ \ \ \ \  & \zeta_{\Gamma_2 '1}=\xi_6, \ \ \ \zeta_{\Gamma_2 '2}=\xi_8,\end{aligned}\end{equation*}
$$\xi_5,\xi_6,\xi_7,\xi_8,\xi_9\in\Omega_{\Gamma_1}, \ \xi_2\in\Omega_{\Gamma_1 '}, \ \ \  \xi_1,\xi_2,\xi_3,\xi_4,\xi_{10}\in\Omega_{\Gamma_2},  \  \xi_7\in\Omega_{\Gamma_2 '}, $$
$$(\operatorname{clos}\mathbb D\setminus\Omega_{\Gamma_1})\cap(\operatorname{clos}\mathbb D\setminus\Omega_{\Gamma_2})=\emptyset,$$
and each of the pairs $\Gamma_1$, $\Gamma_1 '$ and $\Gamma_2$, $\Gamma_2 '$ satisfies \eqref{E:4.2}. 

Applying  Lemma~\ref{L:4.1} to $\Gamma_1$, $\Gamma_1 '$, we obtain 
$\mathcal M_1\in\operatorname{Hlat}T$ such that $\mathcal M_1\neq\{0\}$, and $\sigma(T|_{\mathcal M_1})\subset\mathbb C\setminus\Omega_{\Gamma_1}$. 
Applying  Lemma~\ref{L:4.1} to $\Gamma_2$, $\Gamma_2 '$, we obtain 
$\mathcal M_2\in\operatorname{Hlat}T$ such that $\mathcal M_2\neq\{0\}$, and $\sigma(T|_{\mathcal M_2})\subset\mathbb C\setminus\Omega_{\Gamma_2}$. 
By {\cite[Theorem 0.8]{rara}}, the spectrum of the restriction of an operator on its invariant subspace is contained in the polynomially 
convex hull of  its spectrum. Therefore, $\sigma(T|_{\mathcal M_l})\subset\operatorname{clos}\mathbb D$ $\  (l=1,2)$.  Consequently, 
$\sigma(T|_{\mathcal M_l})\subset\operatorname{clos}\mathbb D\setminus\Omega_{\Gamma_l}$ $\  (l=1,2)$. \end{proof}

The investigation of a.c. polynomially bounded operators $T$ such that $\theta(T)$ is invertible 
for some inner function $\theta\in H^\infty$ is started in {\cite[Proposition 4.4 and Theorem 5.4]{ker11}}. In particular, it is proved in {\cite[Proposition 4.4]{ker11}}  that $\mathbb D\not\subset\sigma(T)$ for such $T$. In {\cite[Theorem 5.4]{ker11}}
 it is shown that if $\theta(T)$ is invertible for an inner function $\theta$, then $\theta(T)^{-1}$ belongs to the closure in the weak operator topology of the algebra of rational functions of $T$. 
 Note that the assumption of  quasianalyticity of $T$ from {\cite[Proposition 4.4 and Theorem 5.4]{ker11}} actually is not used in their proofs, and the proofs work for  a.c. polynomially bounded operators instead of a.c. contractions with  minimal changes.

\begin{lemma}\label{L:4.7} Suppose that $T\in\mathcal L(\mathcal H)$ is an a.c. polynomially bounded operator, $\varphi\in H^\infty$,
 and  $\varphi(T)$ is invertible. 
Then for every $0<c<1/\|\varphi(T)^{-1}\|$ there exists  $C>0$ {\rm (}which also depends on $\varphi$ and $T${\rm)}  such that 
$$\|(T-\lambda)^{-1}(I-\overline\lambda T)\| <  C\ \ \ \text{ for  all }   \lambda\in\mathbb D  \text{ such that }
\ |\varphi(\lambda)|\leq c.$$
Consequently, $$\|(T-\lambda I) ^{-1}x\|\leq C\|(I-\overline\lambda T)^{-1}x\| \ \ \text{ and }\ \  \|(T-\lambda I) ^{-1}\|  < CM/(1-|\lambda|) $$ for all $ x\in\mathcal H$  and all $ \lambda\in\mathbb D $  such that $|\varphi(\lambda)|\leq c$
{\rm (}where $M$ is the polynomial bound of $T${\rm )}.\end{lemma}
\begin{proof} Let $a\in\mathbb C$, and let $|a|\leq c$. Then $$(\varphi(T)-a)^{-1}=\varphi(T)^{-1}(I-a \varphi(T)^{-1})^{-1},$$ therefore,
 \begin{align}\label{E:4.5}\|(\varphi(T)-a)^{-1}\| < \frac{1}{c}\cdot\frac{1}{1-c\|\varphi(T)^{-1}\|}. \end{align}
Suppose that $\lambda\in\mathbb D$, and $|\varphi(\lambda)|\leq c$. Then $\varphi-\varphi(\lambda)=\beta_\lambda \psi$, where 
$\beta_\lambda(z)=\frac{z-\lambda}{1-\overline\lambda z}$  $\ (z\in\mathbb D)$, $\psi\in H^\infty$, and
\begin{align}\label{E:4.6}\|\psi\|_\infty=\|\beta_\lambda\psi\|_\infty=\|\varphi-\varphi(\lambda)\|_\infty
\leq \|\varphi\|_\infty+c. \end{align}
 Since \begin{align}\label{E:4.8}
\beta_\lambda(T)^{-1}=(\varphi(T)-\varphi(\lambda))^{-1}\psi(T),\end{align} we obtain from \eqref{E:4.5}, \eqref{E:4.6}, and \eqref{E:4.8}  that
\begin{equation*} 
\|\beta_\lambda(T)^{-1}\| < M\frac{ \|\varphi\|_\infty+c}{c(1-c\|\varphi(T)^{-1}\|)}=:C.\end{equation*}
The remaining part of the conclusion of the lemma follows from the equality 
\begin{equation*}
(T-\lambda I)^{-1}=\beta_\lambda(T)^{-1}(I-\overline\lambda T)^{-1}.\end{equation*} 
 \end{proof}

\begin{remark}\label{spectral} Lemma \ref{L:4.7} can be considered as a certain type of spectral mapping theorem, 
cf. {\cite[Proof of Lemma 3.1]{bercfp}},  {\cite[2.1(ii)]{makarov}}. 
It is well known that 
\begin{equation}\label{gt}  \sigma(\varphi(T))=\varphi(\sigma(T)),\end{equation}
if $\varphi$ is a function analytic in a neighbourhood of $\sigma(T)$ and $\varphi(T)$ is defined by the Riesz--Dunford functional calculus
for any (linear, bounded) operator $T$ (see, for example, {\cite[Proposition 2.5]{rara}}, {\cite[Theorem 1.2.25]{chalpart}}). 
Furthermore, \eqref{gt} holds if $T$ is a polynomially bounded operator and $\varphi$ is a function from $H^\infty$ which has 
 continuous extension on $\sigma(T)$  {\cite[Corollary 3.2]{foiasmlak}}; see \cite{jkp} for another proof. 
(For detailed definition of $\varphi(T)$, see  \cite{jkp}; decomposition on a.c. and singular parts of polynomially bounded operator is used.)
In particular, \eqref{gt} holds if $T$ is a polynomially bounded operator and $\varphi$ is a function from the disc algebra 
(that is, $\varphi$ is analytic in $\mathbb D$
 and has continuous extension on $\mathbb T$), for proof see also {\cite[Proposition IX.2.4]{sznagy}}. 
But for $H^\infty$-functional calculus for a.c. polynomially bounded operators (and even for a.c. contractions) 
it is impossible to describe   $\sigma(\varphi(T))$ in terms of $\varphi$ and $\sigma(T)$ only, see, for example,  \cite{bercfp}, \cite{makarov}, \cite{rudol} and references therein.\end{remark}

\begin{theorem}\label{T:4.8} Suppose that $T$ is an a.c. polynomially bounded operator and  $\mu$ is a finite positive  singular Borel measure on $\mathbb T$ such that $\mu(\sigma(T)\cap\mathbb T)>0$ and $\mu(\{\zeta\})=0$ 
for every $\zeta\in\sigma(T)\cap\mathbb T$. Suppose that $\theta_\mu(T)$ is invertible. 
 Then there exist $\mathcal M_1$, 
$\mathcal M_2\in\operatorname{Hlat}T$ such that $\mathcal M_1\neq\{0\}$, 
$\mathcal M_2\neq\{0\}$, and $\sigma(T|_{\mathcal M_1})\cap\sigma(T|_{\mathcal M_2})=\emptyset$.  
\end{theorem}
\begin{proof} By {\cite[Theorem II.6.2]{garn}}, $\theta_\mu(z)\to 0$ when $z\to\zeta$ nontangentially for a.e. $\zeta\in\mathbb T$ 
with respect to 
$\mu$. Since $\mu(\sigma(T)\cap\mathbb T)>0$ and $\mu(\{\zeta\})=0$ for every $\zeta\in\sigma(T)\cap\mathbb T$, the set 
$$\mathcal Z_0=\bigl\{\zeta\in\sigma(T)\cap\mathbb T: \ \ \theta_\mu(z)\to 0 \ \text{ when } \ z\to\zeta \ \text{ nontangentially}\bigr\}$$ 
is uncountable.  Let $0<c<1/\|\theta_\mu(T)^{-1}\|$, and let  $C$ be defined by $c$ in Lemma~\ref{L:4.7}.  Set 
 $$\Lambda =\sigma(T)\cup\bigl\{\lambda\in\mathbb D: \ \ \|(T-\lambda I)^{-1}\|\geq CM/(1-|\lambda|)\bigr\}$$ and
$$\mathcal Z=\bigl\{\zeta\in\sigma(T)\cap\mathbb T: \ \ \sup\{r\in [0,1): r\zeta\in\Lambda\}<1\bigr\}.$$
By Lemma~\ref{L:4.7}, we have $\mathcal Z_0\subset\mathcal Z$ and so $\mathcal Z$ is uncountable.
 The conclusion of the theorem follows from Lemma~\ref{L:4.2}. 
 \end{proof}

For $\zeta_1$, $\zeta_2\in\mathbb T$, $\zeta_1\neq\zeta_2$, denote by $J_{\zeta_1,\zeta_2}$ the closed subarc of $\mathbb T$ 
with endpoints $\zeta_1$, $\zeta_2$ which connects $\zeta_1$ with $\zeta_2$ in counter clockwise order.

\begin{theorem}\label{tcircle} Suppose that $T$ is an a.c. polynomially bounded operator, $\sigma(T)=\mathbb T$, 
$\zeta_1$, $\zeta_2\in\mathbb T$, $\zeta_1\neq\zeta_2$,  $\mu$ is a finite positive singular Borel measure on $\mathbb T$ 
 such that $\mu(\{\zeta_l\})>0$ for 
$l=1,2$, and $\theta_\mu(T)$ is invertible.  Then there exist $\mathcal M$, 
$\mathcal M'\in\operatorname{Hlat}T$ such that $\mathcal M\neq\{0\}$, 
$\mathcal M'\neq\{0\}$,  $$\sigma(T|_{\mathcal M})\cap\mathbb T\subset J_{\zeta_1,\zeta_2} \text{ and }
\sigma(T|_{\mathcal M'})\cap\mathbb T\subset J_{\zeta_2,\zeta_1}.$$
\end{theorem}
\begin{proof} Let $0<c<1/\|\theta_\mu(T)^{-1}\|$. By Lemma \ref{adelta}, there exists $0<r<1$ such that 
$$D(r\zeta_l)\subset\{\lambda\in\mathbb D \ :\  |\theta_\mu(\lambda)|\leq c\} \ \ (l=1,2).$$ 
Taking into account Lemma \ref{L:4.7} and the assumption that $\sigma(T)\subset\mathbb T$, 
we conclude that there exist two simple closed rectifiable curves $\Gamma$ and $\Gamma '$ such that 
$\Gamma$ and $\Gamma '$ satisfy \eqref{E:4.1} with $$\zeta_{\Gamma 1}=\zeta_{\Gamma '2}=\zeta_2, \ 
\zeta_{\Gamma 2}=\zeta_{\Gamma '1}=\zeta_1, \ \ J_{\zeta_2,\zeta_1}\setminus\{\zeta_1,\zeta_2\}\subset\Omega_\Gamma, \ J_{\zeta_1,\zeta_2}\setminus\{\zeta_1,\zeta_2\}\subset\Omega_{\Gamma '},$$
  and $$(\Gamma\cup\Gamma ')\cap\mathbb D\subset\{\lambda\in\mathbb D \ : \|(T-\lambda I) ^{-1}\|\leq C/(1-|\lambda|)\}$$ 
for some $C>0$.
 It is easy to see that $\Gamma$, $\Gamma '$ and $T$ satisfy the assumptions of Lemma \ref{L:4.1}. 
The conclusion of the theorem follows from the conclusion of Lemma \ref{L:4.1}.\end{proof}

\begin{theorem}\label{tarc} Suppose that $T$ is an a.c. polynomially bounded operator, 
$\zeta_1$, $\zeta_2\in\mathbb T$, $\zeta_1\neq\zeta_2$,  $\sigma(T)=J_{\zeta_1,\zeta_2}$, 
$\zeta_0\in J_{\zeta_1,\zeta_2}\setminus\{\zeta_1,\zeta_2\}$,  $\mu$ is a finite positive singular Borel measure on $\mathbb T$ 
 such that $\mu(\{\zeta_0\})>0$, and $\theta_\mu(T)$ is invertible.  Then there exist $\mathcal M$, 
$\mathcal M'\in\operatorname{Hlat}T$ such that $\mathcal M\neq\{0\}$, 
$\mathcal M'\neq\{0\}$,  $$\sigma(T|_{\mathcal M})\subset J_{\zeta_0,\zeta_2} \text{ and }
\sigma(T|_{\mathcal M'})\subset J_{\zeta_1,\zeta_0}.$$
\end{theorem}
\begin{proof} Let $0<c<1/\|\theta_\mu(T)^{-1}\|$. By Lemma \ref{adelta}, there exists $0<r<1$ such that 
$$D(r\zeta_0)\subset\{\lambda\in\mathbb D \ :  \ |\theta_\mu(\lambda)|\leq c\}.$$ 
Taking into account Lemma \ref{L:4.7} and the assumption that  $\sigma(T)=J_{\zeta_1,\zeta_2}$, we conclude that there exist 
 two simple closed rectifiable curves $\Gamma$ and $\Gamma '$ such that 
$\Gamma$ and $\Gamma '$ satisfy \eqref{E:4.1} with 
$$ \zeta_{\Gamma 2}=\zeta_{\Gamma '1}=\zeta_0,  \ \ \   
 J_{\zeta_1,\zeta_0}\setminus\{\zeta_0\}\subset\Omega_\Gamma, 
\ J_{\zeta_0,\zeta_2}\setminus\{\zeta_0\}\subset\Omega_{\Gamma '},$$
  and $$(\Gamma\cup\Gamma ')\cap\mathbb D\subset\{\lambda\in\mathbb D \ : \|(T-\lambda I) ^{-1}\|\leq C/(1-|\lambda|)\}$$ 
for some $C>0$.
It is easy to see that $\Gamma$, $\Gamma '$ and $T$ satisfy the assumptions of Lemma \ref{L:4.1}. Let $\mathcal M$ and $\mathcal M'$ 
be the subspaces from the conclusion of Lemma \ref{L:4.1}. We have that $\sigma(T|_{\mathcal M})\subset\mathbb C\setminus\Omega_\Gamma$,  and 
$\sigma(T|_{\mathcal M '})\subset\mathbb C\setminus\Omega_{\Gamma '}$.
By {\cite[Theorem 0.8]{rara}}, the spectrum of the restriction of an operator on its invariant subspace is contained in the polynomially 
convex hull of  its spectrum. Therefore, $\sigma(T|_{\mathcal M})$, $\sigma(T|_{\mathcal M '})\subset J_{\zeta_1,\zeta_2}$. 
Consequently, $\sigma(T|_{\mathcal M})\subset (\mathbb C\setminus\Omega_\Gamma)\cap J_{\zeta_1,\zeta_2}=J_{\zeta_0,\zeta_2}$. Similarly, 
$\sigma(T|_{\mathcal M'})\subset J_{\zeta_1,\zeta_0}$.\end{proof}

\begin{remark}\label{remarkexample}Let $T$ be an a.c. unitary operator, and let $\theta$ be an inner function. 
 Then $\theta(T)$ is a unitary operator, in particular, $\theta(T)$ is invertible. In the next section,  
 examples of nonunitary operators satisfying the assumptions of Theorem \ref{T:4.8} are given (Example \ref{exa28}).  
 \end{remark}

\begin{remark}\label{remarksupp}As it will turn out in the beginning of the next section, operators satisfying the assumptions of Theorem \ref{T:4.8} 
cannot be  quasianalytic. But if one replaces the assumption 
``$\mu(\sigma(T)\cap\mathbb T)>0$" of Theorem \ref{T:4.8} by 
``$\operatorname{supp}\mu\cap\sigma(T)\neq\emptyset$",   where $\operatorname{supp}\mu$  is the \emph{closed} support of $\mu$, then the operator 
satisfying  these modified assumptions  can be  quasianalytic, see Example \ref{exa27} below.
 \end{remark}

\section{Quasianalyticity and invertibility of inner functions of operators}

For the definition of quasianalyticity  we refer to \cite{ker16} or \cite{ker01}  
and to \cite{esterle},   \cite{ker01}, \cite{kerszalai14},  \cite{kerszalai15},  \cite{ker16}, \cite{gamal}
 for examples of quasianalytic  operators. 
We recall only that the set $\pi(T)$ of quasianalyticity of an operator $T$ is a Borel set, $\pi(T)\subset\sigma(T)\cap\mathbb T$, and 
if $T$ is a quasianalytic operator, then $\pi(T)\neq\emptyset$. Furthermore, 
$\pi(T)\subset\sigma(T|_{\mathcal M})$ for every nonzero  $\mathcal M\in\operatorname{Lat}T$ for an a.c. polynomially bounded operator $T$ {\cite[Proposition 35]{ker16}}.  
Therefore, 
operators satisfying the assumptions of Theorem \ref{T:4.8} cannot be quasianalytic. In  this section, examples of quasianalytic 
operators $Q$ are given such that $\theta_\mu(Q)$ is invertible with $\mu(\sigma(Q)\cap\mathbb T)>0$ for a  pure atomic measure $\mu$,  
and also $B(Q)$ is invertible for a Blaschke product $B$ with $\{\zeta\in\mathbb T \ :$ $\ \zeta$ is  a nontangential  limit of zeros of $B\}\subset\sigma(Q)$.

We need the following simple lemma. Recall that for $\lambda\in\mathbb D$ the set $D(\lambda)$ is defined in \eqref{drzeta}. 

\begin{lemma}\label{lemmag} Let $\{\zeta_n\}_{n\in \mathbb N}$ be a sequence of points from $\mathbb T$ such that  
$\zeta_n\neq\zeta_l$ for $n\neq l$. Then there exist 
$\{r_n\}_{n\in \mathbb N}$, $\{a_n\}_{n\in \mathbb N}$ and $\Lambda\subset\mathbb D$ 
such that $0<r_n<1$, $a_n>0$, $\sum_{n\in \mathbb N}(1-r_n)<\infty$, 
$\sum_{n\in \mathbb N}a_n<\infty$, 
$$\mathcal G:=\mathbb D\setminus\cup_{n\in \mathbb N}\operatorname{clos}D(r_n\zeta_n)$$ 
is a simply connected domain,  $B:=\prod_{\lambda\in\Lambda} b_\lambda$ is an interpolating Blaschke product, 
 $ \zeta_n$ is  a nontangential  limit of points from $\Lambda$ for every $n\in\mathbb N$, 
$$\inf_{\mathcal G}|\theta_\mu|>0 \ \text{ with  }  \mu=\sum_{n\in \mathbb N}a_n\delta_{\zeta_n},\ \ \ \text{ and } \ \ \inf_{\mathcal G}|B|>0.$$ 
\end{lemma}
\begin{proof} The domain $\mathcal G$ is constructed by induction. Take $\{r_{1,n}\}_{n\in \mathbb N}$ such that 
$0<r_{1,n}<1$ and  $\sum_{n\in \mathbb N}(1-r_{1,n})<\infty$. Take $r_1\geq r_{1,1}$. Since 
$\zeta_2\not\in\operatorname{clos}D(r_1\zeta_1)$, there exists $r_2\geq r_{1,2}$ such that 
$$\operatorname{clos}D(r_1\zeta_1)\cap\operatorname{clos}D(r_2\zeta_2)=\emptyset.$$
Since 
$$\zeta_3\not\in\operatorname{clos}D(r_1\zeta_1)\cup\operatorname{clos}D(r_2\zeta_2),$$
 there exists $r_3\geq r_{1,3}$ such that
$$\operatorname{clos}D(r_3\zeta_3)\cap(\operatorname{clos}D(r_1\zeta_1)\cup\operatorname{clos}D(r_2\zeta_2))=\emptyset,$$
and so on. 

After $\{r_n\}_{n\in \mathbb N}$ has been constructed, take   $\{a_n\}_{n\in \mathbb N}$ such that $a_n>0$ and 
$$\sum_{n\in \mathbb N}a_n\frac{r_n}{1-r_n}<\infty.$$ By Lemma \ref{adelta}, $\theta_\mu$ with $\mu=\sum_{n\in \mathbb N}a_n\delta_{\zeta_n}$ 
satisfies the conclusion of the lemma. 

Before   constructing an interpolating Blaschke product we recall the definition and the needed facts.
Let $\Lambda\subset\mathbb D$ be such that $\sum_{\lambda\in\Lambda}(1-|\lambda|)<\infty$. 
Set $B_\Lambda=\prod_{\lambda\in\Lambda} b_\lambda$. Then $B_\Lambda$ is a Blaschke product with simple zeros. 
$B_\Lambda$ is called interpolating, if for every $\alpha\in\ell^\infty$ 
there exists $\varphi\in H^\infty$ such that $\varphi|_\Lambda=\alpha$. 
 It is known that $B_\Lambda$ is an interpolating Blaschke product if and only if there exists $c>0$ such that 
\begin{equation}\label{einterb}|B_\Lambda(z)|\geq c\inf_{\lambda\in\Lambda}|b_\lambda(z)|\ \ \text{ for every } z\in\mathbb D\end{equation}
(see, for example, {\cite[Secs. C.3.2.15--C.3.2.18]{nik}}). It follows exactly from the definition 
that if $B_\Lambda$ is an interpolating Blaschke product, then  $B_{\Lambda '}$ 
 is an interpolating Blaschke product, too, for every 
$\Lambda '\subset\Lambda$. 

Let $\Lambda\subset\cup_{n\in \mathbb N}[r_n\zeta_n,\zeta_n)$ be such that $B_\Lambda$ 
is an interpolating Blaschke product. It follows from Lemma \ref{lemmabb}, \eqref{einterb} and the definition of $\mathcal G$ that 
 $\inf_{\mathcal G}|B_\Lambda|\geq\frac{c}{3}$.

To construct Blaschke product, take $\{\rho_n\}_{n\in \mathbb N}$ such that $0<\rho_n<1$ 
and $\sup_{n\in \mathbb N}\frac{1-\rho_{n+1}}{1-\rho_n}<1.$ Then $\prod_{n\in \mathbb N}b_{\rho_n\xi_n}$ 
is an interpolating Blaschke product for arbitrary $\{\xi_n\}_{n\in \mathbb N}\subset\mathbb T$, see, for example, 
{\cite[Theorem 9.2]{duren}} or {\cite[Sec. C.3.3.4(c)]{nik}}. Therefore, $\prod_{n\in \frak N}b_{\rho_n\xi_n}$ is an interpolating Blaschke product for 
every $\frak N\subset\mathbb N$. 

We can consider the partition into disjoint classes 
 $$\cup_{n\in \mathbb N}\{\rho_n\}=\cup_{n\in \mathbb N}\cup_{k\in \mathbb N}\{\rho_{n,k}\},$$ 
where $ \rho_{n,k}\to 1$ when $k\to\infty$ for every $n\in\mathbb N$. Set 
$$\Lambda=\bigcup_{n\in \mathbb N}\ \bigcup_{k: \rho_{n,k}\geq r_n}\{\rho_{n,k}\zeta_n\} \ \ \text{ and } 
\ \ B=B_\Lambda=\prod_{\lambda\in\Lambda} b_\lambda.$$
Then $B$ satisfies the conclusion of the lemma. 
\end{proof}

\begin{lemma}\label{lemmapsi} Suppose that $T$ is an a.c. polynomially bounded operator, $\mathcal G\subset\mathbb D$ 
is a simply connected domain, $g\colon\mathbb D\to\mathcal G$ is a conformal mapping, $Q=g(T)$, $\psi\in H^\infty$,  
and $\inf_{\mathcal G}|\psi|>0$. Then $Q$ is an a.c. polynomially bounded operator and $\psi(Q)$ is invertible. \end{lemma}

\begin{proof} $Q$ is an a.c. polynomially bounded operator and $\psi(Q)=(\psi\circ g)(T)$ by {\cite[Proposition 2]{ker15}}. 
Since $1/\psi\circ g\in H^\infty$, it is easily checked that $\psi(Q)^{-1}=(1/\psi\circ g)(T)$.
\end{proof}

\begin{theorem}\label{thm33}  Let $\{\zeta_n\}_{n\in \mathbb N}$ be a sequence of points from $\mathbb T$ such that  
$\zeta_n\neq\zeta_l$ for $n\neq l$.  Then there exist 
$\{a_n\}_{n\in \mathbb N}$, $\Lambda\subset\mathbb D$, and a quasianalytic contraction $Q$  
such that $a_n>0$, 
$\sum_{n\in \mathbb N}a_n<\infty$,  $B:=\prod_{\lambda\in\Lambda} b_\lambda$ is  an interpolating Blaschke product,  
$ \zeta_n$ is  a nontangential  limit of points from $\Lambda$ for every $n\in\mathbb N$, 
 $\cup_{n\in \mathbb N}\{\zeta_n\}\subset\sigma(Q)$, 
$$\theta_\mu(Q)\ \text{ with  }  \mu=\sum_{n\in \mathbb N}a_n\delta_{\zeta_n}\ \ \ \text{ and } \ \ B(Q) \ \text{ are invertible.}$$ 
\end{theorem}
\begin{proof}
Let $\mathcal G$, $\{a_n\}_{n\in \mathbb N}$, $\Lambda\subset\mathbb D$ be constructed in Lemma \ref{lemmag}. 
Let $g\colon\mathbb D\to\mathcal G$ be a conformal mapping. 
Although $\partial\mathcal G$ is not a Jordan curve,  $\partial\mathcal G$  is a rectifiable curve, since 
$\sum_{n\in \mathbb N}(1-r_n)<\infty$. Therefore, $g'\in H^1$, 
see {\cite[Theorem 10.11]{pom75}} and {\cite[Sec. 6.2,6.3]{pom92}}. 
In particular,  $\sum_{k=0}^\infty|\widehat g(k)|<\infty$ by Hardy's inequality, see, for example, {\cite[Corollary of Theorem 3.15]{duren}}. 
Therefore, $g$ is continuous on  $\operatorname{clos}\mathbb D$. Thus, $g$ belongs to the disc algebra. 
 
Let $T$ be a quasianalytic contraction such that $m(g(\pi(T))\cap\mathbb T)>0$. 
 Set $Q=g(T)$. Since $g'\in H^1$, $Q$ is quasianalytic and $\pi(Q)=g(\pi(T))\cap\mathbb T$ by {\cite[Remark 8 and Corollary 13]{ker15}}. If,  for example, $\pi(T)=\mathbb T$, then $\pi(Q)=\mathbb T$. 

 By Lemmas \ref{lemmag} and \ref{lemmapsi}, $\theta_\mu(Q)$ with $\mu=\sum_{n\in \mathbb N}a_n\delta_{\zeta_n}$ 
 and $B(Q)$ are invertible. 

By the spectral mapping theorem for functions from the disc algebra and polynomially bounded operators, $\sigma(Q)=g(\sigma(T))$, 
see Remark \ref{spectral} for references. Therefore, if $g^{-1}(\zeta_n)\cap\sigma(T)\neq\emptyset$ for some $n\in\mathbb N$, 
then $\zeta_n\in\sigma(Q)$. (The set  $g^{-1}(\zeta_n)$ consists of two points from $\mathbb T$ for every  $n\in\mathbb N$.) 
If, for example, 
$\sigma(T)=\operatorname{clos}\mathbb D$, then $\sigma(Q)=\operatorname{clos}\mathcal G$.\end{proof}

In the following examples, a construction similar to the construction from  Theorem \ref{thm33} is used. 

\begin{example}\label{exa27}
Suppose that $E\subset\mathbb T$ is a set satisfying the assumptions of Lemma \ref{lemmaee} and $E$ contains no isolated points.  
Suppose that   $\mu$ is a finite positive Borel measure on $E$ such that $\mu(\{\zeta\})=0$ for every $\zeta\in E$, and 
$\operatorname{supp}\mu=E$ (where $\operatorname{supp}\mu$  is the closed support of $\mu$). 
We have $\mathbb T\setminus E=\cup_{n\in\mathbb N}\mathcal J_{\zeta_{1n},\zeta_{2n}}$. 
Take $1/2<r<1$. Relabeling $\{\mathcal J_{\zeta_{1n},\zeta_{2n}}\}_{n\in\mathbb N}$ if necessarily, we find  
 $N\in \mathbb N$ such that 
$\operatorname{clos}D(r\zeta_{1n})\cap\operatorname{clos}D(r\zeta_{2n})\neq\emptyset$ for every $n\geq N$
 and $\operatorname{clos}D(r\zeta_{1n})\cap\operatorname{clos}D(r\zeta_{2n})=\emptyset$ for every $n < N$.
Let $n\geq N$. Then $$\mathbb D\setminus (\operatorname{clos}D(r\zeta_{1n})\cup\operatorname{clos}D(r\zeta_{2n}))
=\mathcal G_{1n}\cup\mathcal G_{2n},$$
where $\mathcal G_{1n}$, $\mathcal G_{2n}$ are  disjoint simply connected domains, 
$\partial\mathcal G_{1n}$, $\partial\mathcal G_{2n}$ 
are rectifiable Jordan curves, 
\begin{equation}\label{271}\partial\mathcal G_{1n}\cap\mathbb T=\operatorname{clos}\mathcal J_{\zeta_{1n},\zeta_{2n}}\end{equation}
 and $\partial\mathcal G_{2n}\cap\mathbb T=\mathbb T\setminus\mathcal J_{\zeta_{1n},\zeta_{2n}}$. 
It is easy to see that  
\begin{equation}\label{272}\mathcal G_{1n}\subset\mathbb D\setminus\bigcup_{l\in\mathbb N}
 (\operatorname{clos}D(r\zeta_{1l})\cup\operatorname{clos}D(r\zeta_{2l})).
\end{equation}
For $n<N$, it is easy to construct a simply connected domain $\mathcal G_{1n}$ satisfying \eqref{272} and \eqref{271} and such that  
$\partial\mathcal G_{1n}$ is a  rectifiable Jordan curve. For example, let 
$$ \mathcal G_{1n} = \{z\in\mathbb D\ :\ 2r-1<|z|, \frac{z}{|z|}\in \mathcal J_{\zeta_{1n},\zeta_{2n}}\}\setminus 
(\operatorname{clos}D(r\zeta_{1n})\cup\operatorname{clos}D(r\zeta_{2n})).$$
For every $n\in\mathbb N$, let $g_n\colon\mathbb D\to \mathcal G_{1n}$ be a conformal mapping.  
Since $\partial\mathcal G_{1n}$ is a  rectifiable  curve, we have $g_n'\in H^1$, therefore, $g_n$ belongs to the disc algebra
(see the proof of Theorem \ref{thm33} for references). 
 
Let $n\in\mathbb N$. Let $T$ be a quasianalytic contraction such that $$m(g_n(\pi(T))\cap\mathbb T)>0.$$ 
 Set $Q=g_n(T)$. Then $Q$ is quasianalytic and $\pi(Q)=g_n(\pi(T))\cap\mathbb T$. 
If, for example, $\pi(T)=\mathbb T$, then $\pi(Q)=\mathcal J_{\zeta_{1n},\zeta_{2n}}$. (See the proof of Theorem \ref{thm33} for references.) 
Since  $\pi(Q)\subset \sigma(Q)$ (see the references in the beginning of this section), 
we conclude that $\zeta_{1n}$, $\zeta_{2n}\in\sigma(Q)\cap E$. By \eqref{272}, Lemmas \ref{lemmaee} and 
\ref{lemmapsi},  $\theta_\mu(Q)$ is invertible. Thus, $Q$ is an example to Remark \ref{remarksupp}. \end{example}

\begin{example}\label{exa28}
Let $E$ and   $\mu$ be as in Example \ref{exa27}, 
and let $\{\mathcal G_{1n}\}_{n\in\mathbb N}$ and  $\{ g_n\}_{n\in\mathbb N}$ be constructed in Example \ref{exa27}.
Let   $\{T_n\}_{n\in\mathbb N}$ be a family of a.c. uniformly polynomially bounded operators, in particular, 
a.c. contractions. For every $n\in\mathbb N$, set $Q_n=g_n(T_n)$. By \eqref{gt} applied to $T_n$ and $g_n$, 
taking into account that $g_n$ is from the disc algebra, we conclude that $\sigma(Q_n)=g_n(\sigma(T_n))$. 
Furthermore, $(Q_n-a)^{-1}=(1/(g_n-a))(T_n)$ for $a\not\in \operatorname{clos}\mathcal G_{1n}$ and 
$\theta_\mu(Q_n)^{-1}=(1/\theta_\mu\circ g_n)(T_n)$ (see the proof of Lemma \ref{lemmapsi}). By \eqref{272} and Lemma \ref{lemmaee}, 
$\sup_{n\in\mathbb N}\|\theta_\mu(Q_n)^{-1}\|<\infty$. 

Set $$\mathcal G=\bigcup_{n\in\mathbb N} \mathcal G_{1n}\ \ \ \text{ and } 
\ \ \ Q=\bigoplus_{n\in\mathbb N}Q_n.$$ Then $Q$ is an a.c. polynomially bounded operator, 
$\theta_\mu(Q)$ is invertible, and 
$$\operatorname{clos}\bigcup_{n\in\mathbb N}g_n(\sigma(T_n))\subset\sigma(Q)\subset\operatorname{clos}\mathcal G.$$
If $\sigma(T_n)=\operatorname{clos} \mathbb D$ for all $n\in\mathbb N$, then $\sigma(Q)=\operatorname{clos}\mathcal G$. 
Thus, $Q$ is an example to Remark \ref{remarkexample}.
\end{example}

\section{Invertibility of functions of weighted shifts}

Recall the definition of a bilateral weighted shift, see \cite{shields}.
Let $\omega\colon\mathbb Z \to (0,\infty)$ be a function such that $\sup_{n\in\mathbb Z}\frac{\omega(n+1)}{\omega(n)}<\infty$.
Set $$\ell^2_\omega  =\big \{u=\{u(n)\}_{n\in\mathbb Z}:\ \ 
\|u\|_\omega^2=\sum_{n\in\mathbb Z}|u(n)|^2\omega(n)^2<\infty\big\}.$$
The {\it bilateral weighted shift} $S_\omega\in\mathcal L(\ell^2_\omega)$ acts according to the formula 
$$(S_\omega u)(n)=u(n-1), \ \ n\in\mathbb Z,  \ \ u\in\ell^2_\omega.$$ 
If $\omega(n)=1$ for every $n\in\mathbb Z$, then $S_\omega$ is called the simple bilateral shift. Clearly, 
 the simple bilateral shift is a unitary operator, and $S_\omega$ is similar to the simple bilateral shift if and only if 
\begin{equation}\label{estomega} 0<\inf_{n\in\mathbb Z}\omega(n)\leq\sup_{n\in\mathbb Z}\omega(n)<\infty \end{equation} 
(see {\cite[Theorem $2'$]{shields}}). 
Easy computation shows (see {\cite[Corollary of Proposition 7]{shields}}) that 
\begin{equation}\label{shiftpower}\|S_\omega^k\|=\sup_{n\in\mathbb Z}\frac{\omega(n+k)}{\omega(n)} 
\ \ \text{ for every } k\in\mathbb N\cup\{0\}.\end{equation}
Recall that $\sigma(S_\omega)$ is a closed disc or a (may be degenerate) annulus centered at origin {\cite[Theorem 5]{shields}}.
By {\cite[Corollary of Theorem 2]{shields}}, if $S_\omega$ is a power bounded weighted shift ($\sup_{k\in\mathbb N}\|S_\omega^k\|<\infty$), 
then $S_\omega$ is similar to a contractive weighted shift, which is necessarily a.c.. Therefore, $\varphi(S_\omega)$ is well defined for a power bounded weighted shift  
$S_\omega$ and $\varphi\in H^\infty$. In this section it is proved that if $\liminf|\varphi(z)|=0$ when $|z|\to 1$, 
$S_\omega$ is a power bounded weighted shift with 
$\mathbb T\subset\sigma(S_\omega)$,  and 
$\varphi(S_\omega)$ is invertible, then  $S_\omega$ is similar to the simple bilateral shift. 

We need the following simple lemmas. 

\begin{lemma}\label{lemmar}Suppose that $\{\alpha_n\}_{n=0}^\infty$, $\{\beta_n\}_{n=0}^\infty$ are sequences of non-negative numbers, $C>0$,
 $\{r_j\}_j\subset(0,1)$, $r_j\to 1$,  and $$\sum_{n=0}^\infty\alpha_n r_j^n\leq C \sum_{n=0}^\infty\beta_n r_j^n <\infty\ \ \text{ for all } \ j.$$
Then $\limsup_n (C\beta_n-\alpha_n)\geq 0$. \end{lemma}
\begin{proof} Suppose that there exist $\varepsilon>0$ and an index $N$ such that $C\beta_n-\alpha_n<-\varepsilon$ for every $n\geq N+1$.
Then \begin{align*} 0 & \leq \sum_{n=0}^\infty(C\beta_n-\alpha_n)r_j^n = 
\sum_{n=0}^N(C\beta_n-\alpha_n)r_j^n +\sum_{n=N+1}^\infty(C\beta_n-\alpha_n)r_j^n
\\ & \leq \sum_{n=0}^N(C\beta_n-\alpha_n)r_j^n -\varepsilon \sum_{n=N+1}^\infty r_j^n = 
\sum_{n=0}^N(C\beta_n-\alpha_n)r_j^n - \varepsilon \frac{r_j^{N+1}}{1-r_j}.\end{align*}
The right side of this inequality tends to $-\infty$ when $r_j\to 1$, a contradiction.\end{proof}

\begin{lemma}\label{lemmaomega}Suppose that $\omega\colon\mathbb Z \to (0,\infty)$ is such that 
\begin{equation}\label{powerfinite}\sup_{k\geq 0}\sup_{n\in\mathbb Z}\frac{\omega(n+k)}{\omega(n)} <\infty,
\end{equation} 
$C>0$,  $\{r_j\}_j\subset(0,1)$, $r_j\to 1$,   and 
\begin{equation}\label{44}\sum_{n=0}^\infty\omega(-n-1)^2 r_j^n\leq C^2 \sum_{n=0}^\infty\omega(n)^2 r_j^n 
\ \ \text{ for all } \ j.\end{equation}
Then $\omega$ satisfies  \eqref{estomega}.
\end{lemma}
\begin{proof} First, note that \eqref{powerfinite} implies that $\sup_{n\geq 0}\omega(n)<\infty$. It follows from \eqref{44} and 
Lemma \ref{lemmar} that there exist a subsequence $\{n_j\}_j$  and 
a sequence $\{\varepsilon_j\}_j\subset(0,\infty)$ such that $n_j\to+\infty$,  $\varepsilon_j\to 0$, and 
\begin{equation}\label{enj}\omega(-n_j-1)\leq C\omega(n_j)+\varepsilon_j\ \ \text{ for all } j.\end{equation}
In particular,  $\sup_j\omega(-n_j-1)<\infty$. This relation and \eqref{powerfinite} imply that 
$\sup_{n\geq 0}\omega(-n)<\infty$.

If $\inf_{n\geq 0}\omega(-n)=0$, then \eqref{powerfinite} implies that $\omega(n) = 0$ for every $n\in\mathbb Z$. 
If $\inf_{n\geq 0}\omega(n)=0$, then \eqref{powerfinite} implies that $\omega(n)\to 0$ when $n\to+\infty$. In particular, 
$\omega(n_j)\to 0$. Then, by \eqref{enj}, $\omega(-n_j-1)\to 0$. As was mentioned above, this relation with  \eqref{powerfinite} 
 implies that $\omega(n) = 0$ for every $n\in\mathbb Z$. Thus, $\inf_{n\in\mathbb Z}\omega(n)>0$.\end{proof}

\begin{theorem}\label{tshift} Suppose that $S_\omega$ is a power bounded weighted shift, 
$\mathbb T\subset\sigma(S_\omega)$,   $\varphi\in H^\infty$, $\liminf|\varphi(z)|=0$ when $|z|\to 1$, and  
$\varphi(S_\omega)$ is invertible. Then  $S_\omega$ is similar to the simple bilateral shift. 
\end{theorem}
\begin{proof} As was mentioned above, $S_\omega$ is similar to an a.c. contraction ({\cite[Corollary of Theorem 2]{shields}}); thus, $\varphi(S_\omega)$ is well defined. 

Let $0<c<1/\|\varphi(S_\omega)^{-1}\|$. By assumption, there exists $\{\lambda_j\}_j\subset\mathbb D$ 
such that $|\lambda_j|\to 1$ and $|\varphi(\lambda_j)|\leq c$. By Lemma \ref{L:4.7}, $\lambda_j\not\in\sigma(S_\omega)$
for all $j$. Since $\mathbb T\subset\sigma(S_\omega)\subset\operatorname{clos}\mathbb D$ (the latter inclusion is due to power boundedness of $S_\omega$),  
and $\sigma(S_\omega)$ is an annulus  ({\cite[Theorem 5]{shields}}), we conclude that $\sigma(S_\omega)=\mathbb T$.

 Let $u_n\in\ell^2_\omega$, $u_n(k)=0$, $n\neq k$, $u_n(n)=1$ ($n\in\mathbb Z$). Then 
\begin{align*} (S_\omega-\lambda I)^{-1}u_0&=\sum_{n=0}^\infty \lambda^n u_{-n-1}, &\|(S_\omega-\lambda I)^{-1}u_0\|^2&=\sum_{n=0}^\infty |\lambda|^{2n}\omega(-n-1)^2, \\
 (I-\overline\lambda S_\omega)^{-1}u_0&=\sum_{n=0}^\infty \overline\lambda^n u_n, 
& \|(I-\overline\lambda S_\omega)^{-1}u_0\|^2&=\sum_{n=0}^\infty  |\lambda|^{2n}\omega(n)^2 \end{align*}
 for every $\lambda \in\mathbb D$.
 By Lemma \ref{L:4.7}, 
 there exists $C>0$ such that 
$$\|(S_\omega-\lambda I)^{-1}u\|\leq C\|(I-\overline\lambda S_\omega)^{-1}u\|$$ 
for every  $u\in\ell^2_\omega$ and every $\lambda \in\mathbb D$ such that $|\varphi(\lambda)|\leq c$. 
We obtain that $$\sum_{n=0}^\infty |\lambda_j|^{2n}\omega(-n-1)^2\leq C^2\sum_{n=0}^\infty  |\lambda_j|^{2n}\omega(n)^2$$ 
with $|\lambda_j|\to 1$.  Since  $S_\omega$ is  power bounded, $\omega$ satisfies  \eqref{powerfinite} (see \eqref{shiftpower}).  Applying Lemma \ref{lemmaomega} with $r_j=|\lambda_j|^2$ we obtain \eqref{estomega}. As 
was mentioned in the beginning of this section,  
\eqref{estomega} implies the similarity of $S_\omega$ to the simple bilateral shift. 
\end{proof}

\begin{remark}\label{remest} Recall that for every singular inner function $\theta$ there exists a weight $\omega$ such that $\omega(n)=1$ for $n\geq 0$, the  weighted shift $S_\omega$ is a quasianalytic contraction, 
$\sigma(S_\omega)=\mathbb T$, and $\operatorname{clos}\theta(S_\omega)\ell^2_\omega= \ell^2_\omega$ 
{\cite[Theorem 5.9]{esterle}}. \end{remark}

The author is grateful to the referee for careful reading of the paper and correcting many inaccuracies and misprints.

\end{document}